\documentclass[11pt,a4paper]{article}
\usepackage{epsf,epsfig,amssymb,amsfonts,amsgen,amsmath,amstext,amsbsy,amsopn,amsthm,lineno}
\usepackage{color,comment}
\usepackage{graphicx,tikz}
\usepackage{hyperref}
\usepackage{multirow}
\usepackage{cases}

\newcommand\A{\mathrm{A}}   \newcommand\Aut{\mathrm{Aut}}
 \newcommand\bbF{\mathbb{F}} 
\newcommand\C{\mathrm{C}}     \newcommand\Cay{\mathrm{Cay}}

   \newcommand\GL{\mathrm{GL}}  \newcommand\GU{\mathrm{GU}}

 \newcommand\magma{{\sc Magma} }

    \newcommand\PGU{\mathrm{PGU}}   \newcommand\PSL{\mathrm{PSL}}     \newcommand\PSU{\mathrm{PSU}}
\newcommand\Q{\mathrm{Q}}
   
\newcommand\SL{\mathrm{SL}} \newcommand\SO{\mathrm{SO}}    \newcommand\SU{\mathrm{SU}}  \newcommand\Sy{\mathrm{S}}  
\newcommand\T{\mathsf{T}}
\newcommand\Z{\mathbf{Z}} 

\newtheorem{theorem}{Theorem}[section]
\newtheorem{lemma}[theorem]{Lemma}

\newtheorem{conjecture}[theorem]{Conjecture}
\theoremstyle{definition}

\newtheorem{construction}[theorem]{Construction}

\usepackage{mathdots}

\numberwithin{equation}{section}
\allowdisplaybreaks

\title{\textbf{Cubic Graphical Regular Representations of $\mathrm{PSU}_3(q)$}}
\author{Jing Jian Li\footnote{College of Mathematics and Information Science, Guangxi University, Nanning  530004, P. R. China (\texttt{lijjhx@gxu.edu.cn, zhangxq@st.gxu.edu.cn})}, ~Binzhou Xia\footnote{School of Mathematics and Statistics, The University of Melbourne, Parkville, VIC 3010, Australia (\texttt{binzhoux@unimelb.edu.au, zhesz@student.unimelb.edu.au})}, ~Xiao Qian Zhang\footnotemark[1], ~Shasha Zheng\footnotemark[2]}

\date{}

\begin{document}

\maketitle
\openup 0.5\jot

%
%
%
%
%

\begin{abstract}
A graphical regular representation (GRR) of a group $G$ is a Cayley graph of $G$ whose full automorphism group is equal to the right regular permutation representation of $G$. Towards a proof of the conjecture that only finitely many finite simple groups have no cubic GRR, this paper shows that $\mathrm{PSU}_3(q)$ has a cubic GRR if and only if $q\geq4$. Moreover, a cubic GRR of $\mathrm{PSU}_3(q)$ is constructed for each of these $q$.

\medskip


\smallskip

\end{abstract}

\section{Introduction}

Let $G$ be a group whose identity element is denoted as $1$, and let $S$ be a subset of $G$ such that $1\notin S$ and $S^{-1}=S$, where $S^{-1}=\{x^{-1}: x \in S\}$. The \emph{Cayley graph} of $G$ with \emph{connection set} $S$, denoted by $\Cay(G,S)$, is defined as the graph with vertex set $G$ such that $x$ and $y$ are adjacent if and only if $yx^{-1}\in S$. If one identifies $G$ with its right regular representation, then $G$ is a subgroup of the full automorphism group $\Aut(\Cay(G,S))$ of $\Cay(G,S)$. We call $\Cay(G,S)$ a \emph{graphical regular representation} (\emph{GRR} for short) of $G$ if $\Aut(\Cay(G,S))=G$.

The question which finite groups admit GRRs was studied in a series of papers, and eventually a complete characterization was obtained by Godsil in \cite{Godsil1981}. There is also special interest in studying which finite groups admit GRRs of a prescribed valency. In the case of valency three, Fang, Li, Wang and Xu \cite{FLWX2002} conjectured that every finite non-abelian simple group admits a cubic GRR. However, in \cite{XF2016}, Xia and Fang found that $\PSL_2(7)$ is a counterexample to this conjecture. Meanwhile, they proposed the following conjecture in the same paper.

\begin{conjecture}(\cite[Conjecture 4.3]{XF2016})\label{conj1}
Except a finite number of cases, every finite non-abelian simple group has a cubic GRR.
\end{conjecture}

In \cite[Conjecture 1.3]{Spiga2018}, Spiga conjectured that except $\PSL_2(q)$ and a finite number of other cases, every finite non-abelian simple group $G$ contains an element $x$ and an involution $y$ such that $\Cay(G,\{x,x^{-1},y\})$ is a GRR of $G$. However, this conjecture is not true since both $\PSL_3(q)$ and $\PSU_3(q)$ form infinite families of counterexamples. For a Cayley graph $\Cay(G,S)$, let
\[
\Aut(G,S)=\{\alpha\in \Aut(G): S^{\alpha}=S\}
\]
be the group of automorphisms of $G$ stabilizing $S$ setwise. It is easy to see that if $\Cay(G,S)$ is  a GRR of $G$, then we necessarily have $\Aut(G,S)=1$ and $G=\langle S\rangle$. By \cite[Theorem 4, Corollary 6]{Breda2021}, for $G=\PSL_3(q)$ or $\PSU_3(q)$ and any pair of generators $(x,y)$ of $G$ where $y$ is an involution, $\Aut(G,\{x,x^{-1},y\})$ is always nontrivial and therefore the connection set of any cubic GRR of $G$ (if it exists) consists of three involutions. In \cite[Theorem 1.5]{Xia2021}, Xia, Zheng and Zhou showed that Spiga's conjecture (\cite[Conjecture 1.3]{Spiga2018}) can be saved by adding $\PSL_3(q)$ and $\PSU_3(q)$ to the list of exceptional groups.

Note that, by \cite[Theorem 1.3]{XF2016} and \cite[Theorem 1.2]{X2020}, $\PSL_2(q)$ with $q\notin\{2,3,7\}$ and $\PSL_3(q)$ with $q\neq2$ admit cubic GRRs.
Thus, to settle Conjecture \ref{conj1}, the only remaining family of groups that need to be considered is $\PSU_3(q)$.
It is known by \cite{Coxeter1981} that the non-simple group $\PSU_3(2)$ does not admit cubic GRRs, and it is clear that $\PSU_3(3)$ has no cubic GRR since the group cannot be generated by any triple of involutions. Inspired by the work on cubic GRRs of $\PSL_3(q)$ in \cite{X2020}, in this paper, we construct cubic GRRs of $\PSU_3(q)$ for each $q\geq4$ and verify Conjecture \ref{conj1} for $\PSU_3(q)$, which eventually helps to confirm Conjecture \ref{conj1} in \cite{Xia2021}.

\begin{theorem}\label{thm1}
For a prime power $q$, the group $\PSU_3(q)$ has a cubic GRR if and only if $q\geq4$.
\end{theorem}

The rest of this paper is devoted to the proof of Theorem \ref{thm1}, which is divided into two cases where the characteristic is odd (in Section 2) or even (in Section 3). For $q\geq5$ odd, a graph is constructed in Construction~\ref{cons1} and proved to be a GRR in Theorem~\ref{thm2}; for $q\geq4$ even, a graph is constructed in Construction~\ref{cons2} and proved to be a GRR in Theorem~\ref{thm3}.

\section{Odd characteristic}

Let $q=p^f\geq5$ with odd prime $p$. Denote the projection from $\SU_3(q)$ to $\PSU_3(q)$ by $\eta$.

\begin{lemma}\label{exi1}
Let $q=p^f\geq5$ with odd prime $p$ and $I=\left\{0, f/\gcd(3,f), 2f\gcd(3,f)\right\}$. For an element $b$ in $\bbF_{q^2}^\times$ satisfying $b+b^q=1$ and $b\neq b^q$, there exists an element $a$ of order $q-1$ in $\bbF_{q^2}^\times$ such that
\begin{equation}\label{oc1}
(a+a^{-1}+1)^{p^i}\neq a+a^{-1}b^{q+1}\text{ for any } i\in I,
\end{equation}
\begin{equation}\label{oc2}
(a+a^{-1}+1)^{p^i}\neq 1+b^{q+1}\text{ for any } i\in I
\end{equation}
and
\begin{equation}\label{oc3}
a+a^{-1}b^{q+1}\neq 1+b^{q+1}.
\end{equation}
\end{lemma}

\begin{proof}
Let $S$ be the set of elements $a$ of order $q-1$ in $\bbF_{q^2}^\times$ satisfying the conditions \eqref{oc1}--\eqref{oc3}, let $S_0$ be the set of elements of order $q-1$ in $\bbF_{q^2}^\times$, let
\[
S_{1,i}=\{a\in S_0: (a+a^{-1}+1)^{p^i}=a+a^{-1}b^{q+1}\},
\]
\[
S_{2,i}=\{a\in S_0: (a+a^{-1}+1)^{p^i}=1+b^{q+1}\}
\]
for each $i\in I$, and let
\[
S_3=\{a\in S_0: a+a^{-1}b^{q+1}=1+b^{q+1}\}.
\]
It is clear that
\[|S_{1,0}|+|S_{2,0}|+|S_{3}|\leq 5\]
and that $|S_0|\leq \varphi(q-1)$, where $\varphi$ is Euler's totient function.

First assume $\gcd(3,f)=1$. For $q\in\{5, 7, 9, 11, 13\}$, the conclusion $|S|>0$ can be directly verified by computation in \magma \cite{magma}. For $q>13$, we have
\[
|S|\geq|S_0|-(|S_{1,0}|+|S_{2,0}|+|S_{3}|)\geq \varphi(q-1)-5> 0.
\]

Next assume $\gcd(3,f)=3$. By \cite[Lemma 2.1]{PT2015}, we know $\varphi(q-1)>(q-1)^{2/3}$. Note that
\begin{align*}
S_{1,f/3}&\subseteq \{a\in \bbF_{q^2}^\times: (a+a^{-1}+1)^{q^{1/3}}=a+a^{-1}b^{q+1}\}\\
&=\{a\in \bbF_{q^2}^\times: a^{2q^{1/3}}+1+a^{q^{1/3}}=a^{q^{1/3}+1}+a^{q^{1/3}-1}b^{q+1}\}
\end{align*}
and that
\begin{align*}
S_{1,2f/3}&\subseteq\{a\in \bbF_{q^2}^\times: (a+a^{-1}+1)^{q^{2/3}}=a+a^{-1}b^{q+1}\}\\
&=\{a\in \bbF_{q^2}^\times: a+a^{-1}+1=(a+a^{-1}b^{q+1})^{q^{1/3}}\}\\
&=\{a\in \bbF_{q^2}^\times: a^{q^{1/3}+1}+a^{q^{1/3}-1}+a^{q^{1/3}}=a^{2q^{1/3}}+b^{q+1}\}.
\end{align*}
It follows that
\[|S_{1,f/3}|+|S_{1,2f/3}|\leq 2q^{1/3}+2q^{1/3}=4q^{1/3}.\]
Similarly, we obtain
\[|S_{2,f/3}|+|S_{2,2f/3}|\leq 4q^{1/3}.\]
For $q\in\{3^3, 5^3, 7^3\}$, the conclusion $|S|>0$ can be directly verified by computation in \magma \cite{magma}. For $q>7^3$, we have
\[
|S|>(q-1)^{2/3}-5-8q^{1/3}>0.
\]

Hence there exists $a\in S$, which completes the proof.
\end{proof}

\begin{construction}\label{cons1}
Let $q=p^f\geq5$ with odd prime $p$. Let $b$ be an element in $\bbF_{q^2}^\times$ such that $b+b^q=1$ and $b\neq b^q$, and let $a$ be an element of order $q-1$ in $\bbF_{q^2}^\times$ satisfying the conditions \eqref{oc1}--\eqref{oc3} in Lemma \ref{exi1}.
Moreover, let
\[
X=
\begin{pmatrix}
-b^q&b&b^{q+1}\\
1&0&b^q\\
1&1&-b
\end{pmatrix},
\quad Y=
\begin{pmatrix}
0&0&a\\
0&-1&0\\
a^{-1}&0&0
\end{pmatrix},
\quad Z=
\begin{pmatrix}
0&0&1\\
0&-1&0\\
1&0&0
\end{pmatrix},
\]
$x=X^\eta$, $y=Y^\eta$, $z=Z^\eta$, $S=\{x,y,z\}$ and $\Gamma(q)=\Cay(\PSU_3(q),S)$.
\end{construction}

Note in the above construction that $X,Y,Z$ are indeed elements of
\[
\{A\in\SL(3,q^2): \overline{A}^TWA=W\}=\SU_3(q)
\]
where $\overline{(a_{ij})}=(a_{ij}^q)$ and
\[
W=
\begin{pmatrix}
0&0&1\\
0&1&0\\
1&0&0
\end{pmatrix}.
\]

\begin{lemma}\label{lem1}
In the notation of Construction~$\ref{cons1}$, $o(x)=o(y)=o(z)=2$ and $o(yz)=q-1$ .
\end{lemma}

\begin{proof}
It is clear that none of $x$, $y$ and $z$ is the identity and it is straightforward to verify that $X^2=Y^2=Z^2=1$. Since $|\ker(\eta)|=|\Z(\SU_3(q))|=\gcd(3,q+1)$ is coprime to $2$, we have $o(x)=o(y)=o(z)=2$. Note that
\[
YZ=
\begin{pmatrix}
a&0&0\\
0&1&0\\
0&0&a^{-1}
\end{pmatrix}.
\]
Hence $o(yz)=q-1$.
\end{proof}

\begin{lemma}\label{lem2}
Let $H$ be a maximal subgroup of $\SU_3(q)$ with $q>13$ odd. If $H^\eta$ contains a dihedral group of order $2(q-1)$, then either $H$ is reducible or $H=\SO_3(q)\times\C_{\gcd(q+1,3)}$.
\end{lemma}

\begin{proof}
By \cite[Tables~8.5--8.6]{BHR2013}, either $H$ is reducible or $H$ is one of the following groups:
\begin{itemize}
\item[(i)] $\C_{q+1}^2\rtimes\Sy_3$;
\item[(ii)] $\C_{q^2-q+1}\rtimes\C_3$;
\item[(iii)] $\SU_3(q_0).\C_{\gcd((q+1)/(q_0+1),3)}$, where $q=q_0^r$ with $r$ odd prime;
\item[(iv)] $\SO_3(q)\times\C_{\gcd(q+1,3)}$;
\item[(v)] $3_+^{1+2}\rtimes\Q_8.\C_{\gcd(q+1,9)/3}$, where $q=p\equiv2\pmod{3}$;
\item[(vi)] $\PSL_2(7)\times\C_{\gcd(q+1,3)}$, where $q=p\equiv3,5,6\pmod{7}$;
\item[(vii)] $\C_3.\A_6$, where $q=p\equiv11,14\pmod{15}$.
\end{itemize}
Suppose that $H^\eta$ contains a dihedral group of order $2(q-1)$.

If $H$ is as in~(i), then $6(q+1)^2$ is divisible by $2(q-1)$, whence $6((q+1)/2)^2$ is divisible by $(q-1)/2$. Since $(q+1)/2$ is coprime to $(q-1)/2$, we deduce that $6$ is divisible by $(q-1)/2$, which is not possible as $q>13$.

If $H$ is as in~(ii), then $3(q^2-q+1)$ is divisible by $2(q-1)$, which is not possible.

Assume that $H$ is as in~(iii). Note that $|H|=q_0^3(q_0^3+1)(q_0^2-1)\gcd((q_0^r+1)/(q_0+1),3)$.  For $r\geq 11$, we have $2(q-1)=2(q_0^r-1)>|H|$, a contradiction. For $r=3$, $5$ or $7$, we have
\[
\begin{cases}
  |H|\equiv2\gcd((q_0^3+1)/(q_0+1),3)(q_0^2-1)\pmod{q_0^3-1} &\text{ for }r=3, \\
  |H|\equiv-\gcd((q_0^5+1)/(q_0+1),3)(q_0-1)\pmod{q_0^5-1} &\text{ for }r=5,\\
  |H|\equiv-\gcd((q_0^7+1)/(q_0+1),3)(q_0^6-q_0^5+q_0^3-q_0)\pmod{q_0^7-1} &\text{ for }r=7.
\end{cases}
\]
This implies by direct analysis that $q-1=q_0^r-1$ does not divide $|H|$, again a contradiction.

If $H$ is as in~(v), then $2(q-1)$ is coprime to $3$ as $q\equiv2\pmod{3}$. This implies that $2(q-1)$ divides $8$, which is not possible.

If $H$ is as in~(vi) or~(vii), then $H^\eta=\PSL_2(7)$ or $\A_6$, which implies that $H^\eta$ does not contain any dihedral group of order larger than $10$. However, $H^\eta$ contains a dihedral group of order $2(q-1)$ and $q>13$, a contradiction.
\end{proof}

\begin{lemma}\label{lem3}
In the notation of Construction~$\ref{cons1}$, $\langle X,Y,Z\rangle$ is an irreducible subgroup of $\SU_3(q)$.
\end{lemma}

\begin{proof}
Suppose for a contradiction that $\langle X,Y,Z\rangle$ is reducible. It follows that $\langle X,Y,Z\rangle$ stabilizes a subspace of $\bbF_{q^2}^3$ of dimension $1$ or $2$.

First assume that $\langle X,Y,Z\rangle$ stabilizes a subspace of $\bbF_{q^2}^3$ of dimension $1$, say $\langle u\rangle$, where $u=(u_1,u_2,u_3)$ with $u_1,u_2,u_3\in\bbF_{q^2}$. We have $uZ\in\langle u\rangle$, that is, $(u_3,-u_2,u_1)\in\langle(u_1,u_2,u_3)\rangle$. This implies that either
$u_2\neq 0$ and $u_1=-u_3$, or $u_2=0$ and $u_1=\pm u_3$.
Suppose $u_2\neq 0$ and $u_1=-u_3$. Since $uYZ=(au_1,u_2,-a^{-1}u_1)\in\langle u\rangle=\langle (u_1, u_2, -u_1)\rangle$ and $a\neq 1$, we have $u_1=0$.
Moreover, we derive from $uX=(u_2,0,b^qu_2)\in\langle u\rangle=\langle(0,u_2,0)\rangle$ that $u_2=0$, a contradiction.
Thus we have $u_2=0$ and $u_1=\pm u_3$. Note that $uYZ\in\langle u\rangle$, that is, $(au_1,0,a^{-1}u_3)\in \langle (u_1, 0, u_3)\rangle$.
Since $a^2\neq 1$, we conclude that $u_1=u_3=0$, and hence $u=0$, a contradiction.

Next assume that $\langle X,Y,Z\rangle$ stabilizes a subspace of $\bbF_{q^2}^3$ of dimension $2$. It follows that $\langle X^\T,Y^\T,Z^\T\rangle$ stabilizes a subspace of $\bbF_{q^2}^3$ of dimension $1$, suppose $\langle v\rangle$, where $v=(v_1,v_2,v_3)$ with $v_1,v_2,v_3\in\bbF_{q^2}$. We have $vZ^\T \in\langle v\rangle$, that is, $(v_3,-v_2,v_1)\in\langle(v_1,v_2,v_3)\rangle$.
This implies that either $v_2\neq 0$ and $v_1=-v_3$, or $v_2=0$ and $v_1=\pm v_3$.
Suppose $v_2\neq 0$ and $v_1=-v_3$. Since $vY^\T Z^\T=(a^{-1}v_1,v_2,-av_1)\in\langle v\rangle=\langle (v_1,v_2,-v_1)\rangle$ and $a\neq 1$, we have $v_1=0$.
Moreover, we derive from $vX^\T=(bv_2,0,v_2)\in\langle v\rangle=\langle(0,v_2,0)\rangle$ that $v_2=0$, a contradiction.
Thus we have $v_2=0$ and $v_1=\pm v_3$. Note that $vY^\T Z^\T\in\langle v\rangle$, that is, $(a^{-1}v_1,0,av_3)\in\langle (v_1,0,v_3)\rangle$.
Since $a^2\neq 1$, we conclude that $v_1=v_3=0$, and hence $v=0$, again a contradiction.
\end{proof}

\begin{lemma}\label{lem4}
In the notation of Construction~$\ref{cons1}$, $S$ is a generating set of $\PSU_3(q)$.
\end{lemma}

\begin{proof}
It suffices to prove $\langle X,Y,Z\rangle=\SU_3(q)$. For $q\in\{5,7,9,11,13\}$, computation in \magma~\cite{magma} directly verifies the conclusion. Now suppose that $q>13$ and $\langle X,Y,Z\rangle\neq\SU_3(q)$. It means that $\langle X,Y,Z\rangle$ is contained in a maximal subgroup $H$ of $\SU_3(q)$. From Lemma \ref{lem1}, we deduce that $\langle y,z\rangle$ is a dihedral group of order $2(q-1)$. Thus $H^\eta$ contains a dihedral group of order $2(q-1)$. As a consequence of Lemmas \ref{lem2} and \ref{lem3}, we have $H=\SO_3(q)\times\C_{\gcd(q+1,3)}$.

By Lemma \ref{lem1}, we see that $X$, $Y$ and $Z$ all have order divisible by $2$. Since $|H/\SO_3(q)|=\gcd(q+1,3)$ is odd, we conclude that $X,Y,Z\in\SO_3(q)$. This means that there exists $B\in\GL_3(q^2)$ with $B^\T=B$ and $X^\T BX=Y^\T BY=Z^\T BZ=B$. Write
\[
B=
\begin{pmatrix}
b_{11}&b_{12}&b_{13}\\
b_{12}&b_{22}&b_{23}\\
b_{13}&b_{23}&b_{33}
\end{pmatrix}.
\]
We derive from $Z^\T BZ=B$ that
\[
\begin{pmatrix}
b_{33}&-b_{23}&b_{13}\\
-b_{23}&b_{22}&-b_{12}\\
b_{13}&-b_{12}&b_{11}
\end{pmatrix}
=
\begin{pmatrix}
b_{11}&b_{12}&b_{13}\\
b_{12}&b_{22}&b_{23}\\
b_{13}&b_{23}&b_{33}
\end{pmatrix},
\]
which implies that $b_{11}=b_{33}$ and $-b_{23}=b_{12}$. Next we derive from $Y^\T BY=B$ that
\[
\begin{pmatrix}
a^{-2}b_{11}&a^{-1}b_{12}&b_{13}\\
a^{-1}b_{12}&b_{22}&-ab_{12}\\
b_{13}&-ab_{12}&a^2b_{11}
\end{pmatrix}
=
\begin{pmatrix}
b_{11}&b_{12}&b_{13}\\
b_{12}&b_{22}&-b_{12}\\
b_{13}&-b_{12}&b_{11}
\end{pmatrix}.
\]
Since $a\neq 1$ and $a^2\neq1$, we have $b_{12}=b_{11}=0$.
Thus
\[
B=
\begin{pmatrix}
0&0&b_{13}\\
0&b_{22}&0\\
b_{13}&0&0
\end{pmatrix}.
\]
Without loss of generality, we assume
\[
B=
\begin{pmatrix}
0&0&1\\
0&d&0\\
1&0&0
\end{pmatrix}
\]
with some $d\in \bbF_{q^2}^\times$.
Now we derive from $X^\T BX=B$ that
\[
\begin{pmatrix}
-2b^q+d & b-b^q & 2b^{q+1}+db^q\\
b-b^q & 2b & b^{q+1}-b^2\\
2b^{q+1}+db^q & b^{q+1}-b^2 & -2b^{q+2}+db^{2q}
\end{pmatrix}
=
\begin{pmatrix}
0&0&1\\
0&d&0\\
1&0&0
\end{pmatrix},
\]
from which we deduce $b=b^q$, contradicting the conditions for $b$ in Construction \ref{cons1}.
\end{proof}

\begin{lemma}\label{lem5}
In the notation of Construction~$\ref{cons1}$, $\Aut(\PSU_3(q),S)=1$.
\end{lemma}

\begin{proof}
Let $\alpha$ be an arbitrary element of $\Aut(\PSU_3(q),S)$. We aim to show that $\alpha$ fixes every element of $S$. Since Lemma~\ref{lem4} asserts that $S$ generates $\PSU_3(q)$, this will lead to $\alpha=1$ and hence $\Aut(\PSU_3(q),S)=1$.

Recall that the characteristic polynomial of $YZ$ is
\[
\lambda^3-(a+a^{-1}+1)\lambda^2+(a+a^{-1}+1)\lambda-1.
\]
Direct calculation also shows that the characteristic polynomials of
\[
XY=
\begin{pmatrix}
a^{-1}b^{q+1}&-b&-ab^{q}\\
a^{-1}b^q&0&a\\
-a^{-1}b&-1&a
\end{pmatrix}
\quad \text{and} \quad
XZ=
\begin{pmatrix}
b^{q+1}&-b&b^{q}\\
b^q&0&1\\
-b&-1&1
\end{pmatrix},
\]
are
\[
\lambda^3-(a+a^{-1}b^{q+1})\lambda^2+(a+a^{-1}b^{q+1})\lambda-1
\]
and
\[
\lambda^3-(1+b^{q+1})\lambda^2+(1+b^{q+1})\lambda-1
\]
respectively.

Since the order of $\alpha$ is $1$, $2$ or $3$, we deduce that $\alpha\PGU_3(q)$ has order $1$, $2$ or $3$ in
\[
\Aut(\PSU_3(q))/\PGU_3(q)=\langle\phi\PGU_3(q)\rangle,
\]
where $\phi$ is the field automorphism of order $2f$.
Hence $\alpha\PGU_3(q)=\phi^{i}\PGU_3(q)$ for some $i\in\left\{f, 2f, 2f/\gcd(3,f), 4f/\gcd(3,f)\right\}$. By abuse of notation we also denote by $\phi$ the corresponding automorphism of $\SU_3(q)$ mapping each entry of a matrix to its $p$th power.

We first confirm that $\alpha$ fixes $x$ by showing that $(yz)^\alpha\notin\{xy,yx,xz,zx\}$.

(i) Suppose that $(yz)^\alpha=xy$ (or $yx$). It follows that
\[
sD^{-1}(YZ)^{\phi^{i}}D=XY
\]
for some $s\in\bbF_{q^2}^\times$ and $D\in\GU_3(q)$. Hence $s(YZ)^{\phi^{i}}=D(XY)D^{-1}$
has the same characteristic polynomial as $XY$, that is,
\begin{align*}
&\lambda^3-s(a+a^{-1}+1)^{p^{i}}\lambda^2+s^2(a+a^{-1}+1)^{p^{i}}\lambda-s^3\\
=&\lambda^3-(a+a^{-1}b^{q+1})\lambda^2+(a+a^{-1}b^{q+1})\lambda-1.
\end{align*}
We deduce that $(a+a^{-1}+1)^{p^{i}}=a+a^{-1}b^{q+1}$, contradicting \eqref{oc1}.

(ii) Suppose that $(yz)^\alpha=xz$ (or $zx$). Similarly, we have
\begin{align*}
&\lambda^3-t(a+a^{-1}+1)^{p^{i}}\lambda^2+t^2(a+a^{-1}+1)^{p^{i}}\lambda-t^3\\
=&\lambda^3-(1+b^{q+1})\lambda^2+(1+b^{q+1})\lambda-1.
\end{align*}
for some $t\in\bbF_{q^2}^\times$. Hence $(a+a^{-1}+1)^{p^{i}}=1+b^{q+1}$, which contradicts \eqref{oc2}.

To finish the proof, we only need to show that $\alpha$ cannot swap $y$ and $z$.

Suppose for a contradiction that $\alpha$ swaps $y$ and $z$. It implies that $\alpha$ has order $2$, and so $\alpha\PGU_3(q)=\phi^{j}\PGU_3(q)$ for some $i\in\left\{f,2f\right\}$.
As $(xy)^\alpha=x^\alpha y^\alpha=xz$, it follows that
\begin{align*}
&\lambda^3-r(a+a^{-1}b^{q+1})^{p^{i}}\lambda^2+r^2(a+a^{-1}b^{q+1})^{p^{i}}\lambda-r^3\\
=&\lambda^3-r(a+a^{-1}b^{q+1})\lambda^2+r^2(a+a^{-1}b^{q+1})\lambda-r^3\\
=&\lambda^3-(1+b^{q+1})\lambda^2+(1+b^{q+1})\lambda-1.
\end{align*}
for some $r\in\bbF_{q^2}^\times$. Hence $a+a^{-1}b^{q+1}=a+a^{-1}+1$, contradicting \eqref{oc3}.

This shows that $\alpha$ fixes every element of $S$, as desired.
\end{proof}

\begin{theorem}\label{thm2}
For each odd prime power $q\geq5$, the graph $\Gamma(q)$ as in Construction~$\ref{cons1}$ is a cubic GRR of $\PSU_3(q)$.
\end{theorem}

\begin{proof}
By \cite[Theorem~1.3]{FLWX2002}, for any connection set $S$ of size three, the graph $\Cay(\PSU_3(q),S)$ is a GRR of $\PSU_3(q)$ if and only if $\langle S\rangle=\PSU_3(q)$ and $\Aut(\PSU_3(q),S)=1$. Now let $S$ be the set in Construction~\ref{cons1}. Lemma~\ref{lem4} shows that $\langle S\rangle=\PSU_3(q)$, and Lemma~\ref{lem5} shows that $\Aut(\PSU_3(q),S)=1$. Hence $\Gamma(q)=\Cay(\PSU_3(q),S)$ is a cubic GRR of $\PSU_3(q)$.
\end{proof}

\section{Even characteristic}

Let $q=2^f\geq4$. Denote the projection from $\SU_3(q)$ to $\PSU_3(q)$ by $\eta$.

\begin{lemma}\label{exi2}
Let $q=2^f\geq4$ and $I=\left\{0, f, 2f/\gcd(3,f), 4f/\gcd(3,f)\right\}$. For an element $b$ in $\bbF_{q^2}^\times$ satisfying $b+b^q=1$ and $b^{q+1}\neq 1$, there exists an element $a$ of order $q+1$ such that
\begin{equation}\label{ec1}
(1+b+b^2)^{2^{i}}\neq a+a^{-1}+1 \text{ for any } i\in I,
\end{equation}
\begin{equation}\label{ec2}
(1+b+b^2)^{2^{i}}\neq a(1+b+b^3+b^4)+a^{-1}(b^2+b^3+b^4) \text{ for any } i\in I,
\end{equation}
\begin{equation}\label{ec3}
a+a^{-1}+1\neq a(1+b+b^3+b^4)+a^{-1}(b^2+b^3+b^4),
\end{equation}
and
\begin{equation}\label{ec4}
a+a^{-1}+1\neq 0.
\end{equation}
\end{lemma}

\begin{proof}
Let $S$ be the set of elements $a$ of order $q+1$ in $\bbF_{q^2}^\times$ satisfying the conditions \eqref{ec1}--\eqref{ec4}, let $S_0$ be the set of elements of order $q+1$ in $\bbF_{q^2}^\times$, let
\[
S_{1,i}=\{a\in S_0: (1+b+b^2)^{2^{i}}= a+a^{-1}+1\},
\]
\[
S_{2,i}=\{a\in S_0: (1+b+b^2)^{2^{i}}= a(1+b+b^3+b^4)+a^{-1}(b^2+b^3+b^4)\}
\]
for each $i\in I$ and let
\[
S_3=\{a\in S_0: a+a^{-1}+1= a(1+b+b^3+b^4)+a^{-1}(b^2+b^3+b^4)\},
\]
\[
S_4=\{a\in S_0: a+a^{-1}+1= 0\}.
\]
It is clear that $|S_{1,i}|\leq 2$, $|S_{2,i}|\leq 2$, $|S_3|\leq 2$ and $|S_4|\leq 2$.

For $q\in\{2^2, 2^3, 2^4, 2^5\}$, the conclusion $|S|>0$ can be directly verified. Note that $\varphi(q+1)>20$ for $q=2^f\geq 2^{6}$. Thus for $q\geq 2^{6}$, we have
\[
|S|\geq|S_0|-\left(\sum_{i\in I}|S_{1,i}|+\sum_{i\in I}|S_{2,i}|+|S_{3}|+|S_{4}|\right)\geq\varphi(q+1)-20>0.
\]

Hence there exists $a\in S$, which completes the proof.
\end{proof}

\begin{construction}\label{cons2}
Let $q=2^f\geq4$. Let $b$ be an element in $\bbF_{q^2}^\times$ satisfying $b+b^q=1$ and $b^{q+1}\neq 1$, and let $a$ be an element of order $q+1$ satisfying the conditions \eqref{ec1}--\eqref{ec4} in Lemma \ref{exi2}. Moreover, let
\[
X=
\begin{pmatrix}
b&1&1\\
b^q&0&1\\
b^{q+1}&b&b^q
\end{pmatrix},
\quad Y=
\begin{pmatrix}
ab+a^{-1}b^q&0&a(b+b^{q+1})+a^{-1}(b^q+b^{q+1})\\
0&1&0\\
a+a^{-1}&0&ab+a^{-1}b^q
\end{pmatrix},
\]
\[
\quad Z=
\begin{pmatrix}
1&0&1\\
0&1&0\\
0&0&1
\end{pmatrix},
\]
$x=X^\eta$, $y=Y^\eta$, $z=Z^\eta$, $S=\{x,y,z\}$ and $\Gamma(q)=\Cay(\PSU_3(q),S)$.
\end{construction}

Note in the above construction that $X,Y,Z$ are indeed elements of $\SU_3(q)$.

\begin{lemma}\label{lem6}
In the notation of Construction~$\ref{cons2}$, $o(x)=o(y)=o(z)=2$, and $o(zy)=q+1$.
\end{lemma}

\begin{proof}
It is clear that none of $x$, $y$ and $z$ is the identity and it is straightforward to verify that $X^2=Y^2=Z^2=1$. Since $|\ker(\eta)|=|\Z(\SU_3(q))|=\gcd(3,q+1)$ is coprime to $2$, we have $o(x)=o(y)=o(z)=2$. From the characteristic polynomial of
\[
ZY=
\begin{pmatrix}
ab^q+a^{-1}b&0&(a+a^{-1})b^{q+1}\\
0&1&0\\
a+a^{-1}&0&ab+a^{-1}b^q
\end{pmatrix},
\]
which is
\begin{eqnarray*}
&&
\begin{vmatrix}
\lambda+ab^q+a^{-1}b&0&(a+a^{-1})b^{q+1}\\
0&\lambda+1&0\\
a+a^{-1}&0&\lambda+ab+a^{-1}b^q
\end{vmatrix}=(\lambda+a)(\lambda+a^{-1})(\lambda+1),
\end{eqnarray*}
we see that the eigenvalues of $ZY$ are $a$, $a^{-1}$ and $1$, whence $o(zy)=q+1$.
\end{proof}

\begin{lemma}\label{lem7}
Let $H$ be a maximal subgroup of $\SU_3(q)$ with $q>8$ even. If $H^\eta$ contains a dihedral group of order $2(q+1)$, then either $H$ is reducible or $H=\C_{q+1}^2\rtimes\Sy_3$.
\end{lemma}

\begin{proof}
By \cite[Tables~8.5--8.6]{BHR2013}, either $H$ is reducible or $H$ is one of the following groups:
\begin{itemize}
\item[(i)] $\C_{q+1}^2\rtimes\Sy_3$;
\item[(ii)] $\C_{q^2-q+1}\rtimes\C_3$;
\item[(iii)] $\SU_3(q_0).\C_{\gcd((q+1)/(q_0+1),3)}$, where $q=q_0^r$ with $r$ odd prime.
\end{itemize}
Suppose that $H^\eta$ contains a dihedral group of order $2(q+1)$.

If $H$ is as in (ii), then $3(q^2-q+1)$ is divisible by $2(q+1)$, which is not possible.

Assume that $H$ is as in (iii). For $r=3$, we know that $H$ has no dihedral subgroup of order $2(q+1)=2(q_0^3+1)$ according to \cite[Tables~8.5--8.6]{BHR2013}, a contradiction. For $r=5$ or $7$, direct analysis shows that $q+1=q_0^r+1$ does not divide $|H|$ since
\[
\begin{cases}
  |H|\equiv-\gcd((q_0^5+1)/(q_0+1),3)(2q_0^3-q_0+1)\pmod{q_0^5+1} &\text{ for }r=5,\\
  |H|\equiv-\gcd((q_0^7+1)/(q_0+1),3)(q_0^6-q_0^5+q_0^3+q_0)\pmod{q_0^7+1} &\text{ for }r=7,
\end{cases}
\]
a contradiction. For $r\geq 11$, we have $2(q+1)>|H|$, again a contradiction.
\end{proof}

\begin{lemma}\label{lem9}
In the notation of Construction~$\ref{cons2}$, $\langle X,Y,Z\rangle$ is an irreducible subgroup of $\SU_3(q)$.
\end{lemma}

\begin{proof}
Suppose for a contradiction that $\langle X,Y,Z\rangle$ is reducible. It follows that $\langle X,Y,Z\rangle$ stabilizes a subspace of $\bbF_{q^2}^3$ of dimension $1$ or $2$.

First assume that $\langle X,Y,Z\rangle$ stabilizes a subspace of $\bbF_{q^2}^3$ of dimension $1$, say $\langle u\rangle$, where $u=(u_1,u_2,u_3)$ with $u_1,u_2,u_3\in\bbF_{q^2}$. We have $uZ\in\langle u\rangle$, that is, $(u_1,u_2,u_1+u_3)\in\langle(u_1,u_2,u_3)\rangle$. This implies that $u_1=0$.
Note that $uY=((a+a^{-1})u_3,u_2,(ab+a^{-1}b^q)u_3)\in\langle(0,u_2,u_3)\rangle$. In view of $a+a^{-1}\neq 0$, we obtain $u_3=0$.
Moreover, we derive from $uX=(b^{q}u_2,0,u_2)\in\langle(0,u_2,0)\rangle$ that $u_2=0$. Thus $u=0$, a contradiction.

Next assume that $\langle X,Y,Z\rangle$ stabilizes a subspace of $\bbF_{q^2}^3$ of dimension $2$. It follows that $\langle X^\T,Y^\T,Z^\T\rangle$ stabilizes a subspace of $\bbF_{q^2}^3$ of dimension $1$, say $\langle v\rangle$, where $v=(v_1,v_2,v_3)$ with $v_1,v_2,v_3\in\bbF_{q^2}$.
From $vZ^\T=(v_1+v_3,v_2,v_3)\in\langle(v_1,v_2,v_3)\rangle$, we have $v_3=0$.
Note that $vY^\T=((ab+a^{-1}b^q)v_1,v_2,(a+a^{-1})v_1)\in\langle(v_1,v_2,0)\rangle$. Since $a+a^{-1}\neq0$, we obtain $v_1=0$.
Moreover, we see from $vX^\T=(v_2,0,bv_2)\in\langle(0,v_2,0)\rangle$ that $v_2=0$. Hence $v=0$, again a contradiction.
\end{proof}

\begin{lemma}\label{lem8}
In the notation of Construction~$\ref{cons2}$, $S$ is a generating set of $\PSU_3(q)$.
\end{lemma}

\begin{proof}
It suffices to prove $\langle X,Y,Z\rangle=\SU_3(q)$. For $q=4$ or 8, computation in \magma~\cite{magma} directly verifies the conclusion. Now suppose that $q>8$ and $\langle X,Y,Z\rangle\neq\SU_3(q)$. By Lemma~\ref{lem6}, we see that $\langle y,z\rangle$ is a dihedral group of order $2(q+1)$. According to Lemma~\ref{lem7} and Lemma~\ref{lem9}, we have $\langle X,Y,Z\rangle=\C_{q+1}^2\rtimes\Sy_3$.
This implies that $\langle X,Y,Z\rangle$ stabilizes a direct sum decomposition of $V=V_1\oplus V_2\oplus V_3$ such that each $V_i$ is a non-singular subspace of dimension $1$.

Note that $X$, $Y$ and $Z$ are all involutions, which means that each of them stabilizes at least one of the three subspaces.
For a given vector $(u_1,u_2,u_3)$ with $u_1,u_2,u_3\in\bbF_{q^2}$, we have
\[
(u_1,u_2,u_3)Z=(u_1,u_2,u_1+u_3).
\]
It is easy to see that $Z$ stabilizes $\langle(u_1,u_2,u_3)\rangle$ if and only if $u_1=0$. Thus, for the three subspaces of the decomposition, $Z$ swaps two of them, suppose $V_1$ and $V_2$, and stabilizes the third one, $V_3$. Let $V_1=\langle (1,k_2,k_3) \rangle$, $V_2=\langle (1,k_2,1+k_3)\rangle$ and $V_3=\langle (0,k_4,k_5) \rangle$ with $k_2,k_3,k_4,k_5 \in\bbF_{q^2}$. It is clear that $k_4\neq 0$, otherwise $(V_1\oplus V_2)\cap V_3\neq\{0\}$.

We claim that $k_5\neq0$. Suppose for a contradiction that $k_5=0$.

Assume $V_3X=V_3$. Hence
\[
(b^qk_4, 0, k_4)\in\langle(0, k_4, 0)\rangle.
\]
Thus we have $k_4=0$, a contradiction.

Assume that $X$ stabilizes $V_1$, and swaps $V_2$ and $V_3$. Hence
\begin{numcases}{}
(b+b^qk_2+b^{q+1}k_3, 1+bk_3, 1+k_2+b^qk_3)\in\langle(1, k_2, k_3)\rangle,  \label{dsd1}\\
(b+b^qk_2+b^{q+1}(1+k_3), 1+b(1+k_3), 1+k_2+b^q(1+k_3))\in\langle(0, k_4, 0)\rangle.  \label{dsd2}
\end{numcases}
It follows from \eqref{dsd2} that
\[
\begin{cases}
b+b^qk_2+b^{q+1}(1+k_3)=0,\\
1+k_2+b^q(1+k_3)=0.
\end{cases}
\]
This implies that $k_2=0$ and
\begin{equation}\label{dsd3}
1+b^q(1+k_3)=0.
\end{equation}
Next we see from \eqref{dsd1} that $1+bk_3=0$. This together with \eqref{dsd3} implies that $k_3=b^{q}$, and so $b^{q+1}=1$, a contradiction.

Assume that $X$ stabilizes $V_2$, and swaps $V_1$ and $V_3$. Hence
\begin{numcases}{}
(b+b^qk_2+b^{q+1}k_3, 1+bk_3, 1+k_2+b^qk_3)\in\langle(0, k_4, 0)\rangle, \label{dsd4}\\
(b+b^qk_2+b^{q+1}(1+k_3), 1+b(1+k_3), 1+k_2+b^q(1+k_3))\in\langle(1, k_2, 1+k_3)\rangle. \label{dsd5}
\end{numcases}
It follows from \eqref{dsd4} that
\[
\begin{cases}
b+b^qk_2+b^{q+1}k_3=0,\\
1+k_2+b^qk_3=0.
\end{cases}
\]
This implies that $k_2=0$ and
\begin{equation}\label{dsd6}
1+b^qk_3=0.
\end{equation}
Next we see from \eqref{dsd5} that $1+b(1+k_3)=0$. This together with \eqref{dsd6} implies that $k_3=b$, and so $b^{q+1}=1$, again a contradiction.

Thus we conclude that $k_5\neq 0$.

Since $a+a^{-1}\neq0$, we have
\[ ((a+a^{-1})k_5, k_4, (ab+a^{-1}b^q)k_5)\notin\langle(0, k_4, k_5)\rangle.\]
This implies that $V_3Y\neq V_3$. Thus $V_3Y=V_1 \text{ or } V_2$.

If $V_3Y=V_1$ and $V_2Y=V_2$, then
\begin{numcases}{}
((a+a^{-1})k_5, k_4, (ab+a^{-1}b^q)k_5) \in\langle(1, k_2, k_3)\rangle,\label{dsd7}\\
(ab^q+a^{-1}b+(a+a^{-1})k_3, k_2, (a+a^{-1})b^{q+1}+(ab+a^{-1}b^q)k_3)\label{dsd8}
\in\langle(1, k_2, 1+k_3)\rangle.
\end{numcases}
It follows from \eqref{dsd7} that $k_2\neq 0$ and
\begin{equation}\label{dsd9}
(a+a^{-1})k_3=ab+a^{-1}b^q.
\end{equation}
Next we see from \eqref{dsd8} that $ab^q+a^{-1}b+(a+a^{-1})k_3=1$.
Combining this with \eqref{dsd9}, we have $a+a^{-1}+1=0$, which contradicts \eqref{ec4}.

If $V_3Y=V_2$ and $V_1Y=V_1$, then
\[
\begin{cases}
((a+a^{-1})k_5, k_4, (ab+a^{-1}b^q)k_5) \in\langle(1, k_2, 1+k_3)\rangle,\\
(ab+a^{-1}b^q+(a+a^{-1})k_3, k_2, a(b+b^{q+1})+a^{-1}(b^q+b^{q+1})+(ab+a^{-1}b^q)k_3)
\in\langle(1, k_2, k_3)\rangle.
\end{cases}
\]
Similarly, we have
\[
\begin{cases}
(a+a^{-1})(1+k_3)=ab+a^{-1}b^q,\\
ab+a^{-1}b^q+(a+a^{-1})k_3=1.
\end{cases}
\]
This implies that $a+a^{-1}+1=0$, again contradicting \eqref{ec4}. This completes the proof.
\end{proof}

\begin{lemma}\label{lem10}
In the notation of Construction~$\ref{cons1}$, $\Aut(\PSU_3(q),S)=1$.
\end{lemma}

\begin{proof}
Let $\alpha$ be an arbitrary element of $\Aut(\PSU_3(q),S)$. We aim to show that $\alpha$ fixes every element of $S$. Since Lemma~\ref{lem8} asserts that $S$ generates $\PSU_3(q)$, this will lead to $\alpha=1$ and hence $\Aut(\PSU_3(q),S)=1$.

Recall that the characteristic polynomial of $ZY$ is
\[
\lambda^3+(a+a^{-1}+1)\lambda^2+(a+a^{-1}+1)\lambda+1.
\]
Direct calculation also shows that the characteristic polynomials of $ZX$ and $XY$ are
\[
\lambda^3+(1+b+b^2)\lambda^2+(1+b+b^2)\lambda+1
\]
and
\[
\lambda^3+\left( a(1+b+b^3+b^4)+a^{-1}(b^2+b^3+b^4) \right)\lambda^2+\left( a(1+b+b^3+b^4)+a^{-1}(b^2+b^3+b^4) \right)\lambda+1
\]
respectively.

Since the order of $\alpha$ is $1$, $2$ or $3$, we deduce that $\alpha\PGU_3(q)$ has order $1$, $2$ or $3$ in
\[
\Aut(\PSU_3(q))/\PGU_3(q)=\langle\phi\PGU_3(q)\rangle,
\]
where $\phi$ is the field automorphism of order $2f$.
Hence $\alpha\PGU_3(q)=\phi^{i}\PGU_3(q)$ for some $i\in\left\{f, 2f, 2f/\gcd(3,f), 4f/\gcd(3,f)\right\}$. By abuse of notation we also denote by $\phi$ the corresponding automorphism of $\SU_3(q)$ mapping each entry of a matrix to its $2$nd power.

We first confirm that $\alpha$ fixes $y$ by showing that $(zx)^\alpha\notin\{xy,yx,yz,zy\}$.

(i) Suppose that $(zx)^\alpha=xy$ (or $yx$). It follows that
\[
sD^{-1}(ZX)^{\phi^{i}}D=XY
\]
for some $s\in\bbF_{q^2}^\times$ and $D\in\GU_3(q)$. Hence $s(ZX)^{\phi^{i}}=D(XY)D^{-1}$ has the same characteristic polynomial as $XY$, that is,
\begin{align*}
&\lambda^3+s(1+b+b^2)^{2^{i}}\lambda^2+s^2(1+b+b^2)^{2^{i}}\lambda+s^3\\
=&\lambda^3+\left( a(1+b+b^3+b^4)+a^{-1}(b^2+b^3+b^4) \right)\lambda^2\\
&+\left( a(1+b+b^3+b^4)+a^{-1}(b^2+b^3+b^4) \right)\lambda+1.
\end{align*}
Thereby we deduce that
\[
(1+b+b^2)^{2^{i}}=a(1+b+b^3+b^4)+a^{-1}(b^2+b^3+b^4),
\]
contradicting \eqref{ec2}.

(ii) Suppose that $(zx)^\alpha=yz$ (or $zy$). Similarly, we have
\begin{align*}
&\lambda^3+s(1+b+b^2)^{2^{i}}\lambda^2+s^2(1+b+b^2)^{2^{i}}\lambda+t^3\\
=&\lambda^3+(a+a^{-1}+1)\lambda^2+(a+a^{-1}+1)\lambda+1.
\end{align*}
for some $t\in\bbF_{q^2}^\times$. Hence
\[
(1+b+b^2)^{2^{i}}=a+a^{-1}+1,
\]
which contradicts \eqref{ec1}.

To finish the proof, we only need to show that $\alpha$ cannot swap $x$ and $z$.

Suppose for a contradiction that $\alpha$ swaps $x$ and $z$. It implies that $\alpha$ has order $2$, and so $\alpha\PGU_3(q)=\phi^{i}\PGU_3(q)$ for some $i\in\left\{f,2f\right\}$. As $(zy)^\alpha=z^\alpha y^\alpha=xy$, it follows that
\begin{align*}
&\lambda^3+r(a+a^{-1}+1)^{2^{i}}\lambda^2+r^2(a+a^{-1}+1)^{2^{i}}\lambda+r^3\\
=&\lambda^3+\left( a(1+b+b^3+b^4)+a^{-1}(b^2+b^3+b^4) \right)\lambda^2\\
&+\left( a(1+b+b^3+b^4)+a^{-1}(b^2+b^3+b^4) \right)\lambda
+1.
\end{align*}
for some $r\in\bbF_{q^2}^\times$. Thereby we deduce that
\[
(a+a^{-1}+1)^{2^{i}}=a(1+b+b^3+b^4)+a^{-1}(b^2+b^3+b^4),
\]
which is equivalent to
\[
a+a^{-1}+1=a(1+b+b^3+b^4)+a^{-1}(b^2+b^3+b^4),
\]
contradicting \eqref{ec3}.

This shows that $\alpha$ fixes every element of $S$, as desired.
\end{proof}

\begin{theorem}\label{thm3}
For every $2$-power $q\geq4$, the graph $\Gamma(q)$ as in Construction~$\ref{cons1}$ is a cubic GRR of $\PSU_3(q)$.
\end{theorem}

\begin{proof}
By~\cite[Theorem~1.3]{FLWX2002}, for any connection set $S$ of size three, the graph $\Cay(\PSU_3(q),S)$ is a GRR of $\PSU_3(q)$ if and only if $\langle S\rangle=\PSU_3(q)$ and $\Aut(\PSU_3(q),S)=1$. Now let $S$ be the set in Construction~\ref{cons2}. Lemma~\ref{lem8} shows that $\langle S\rangle=\PSU_3(q)$, and Lemma~\ref{lem10} shows that $\Aut(\PSU_3(q),S)=1$. Hence $\Gamma(q)=\Cay(\PSU_3(q),S)$ is a cubic GRR of $\PSU_3(q)$.
\end{proof}

\medskip

\noindent \textbf{Acknowledgements}~~This work was partially supported by NNSFC (12061092), Yunnan Applied Basic Research Projects (202101AT070137) and Melbourne Research Scholarship.

\end{document}